\documentclass[11pt]{amsart}
\usepackage[colorlinks=true,pagebackref,hyperindex,citecolor=green,linkcolor=red]{hyperref}
\usepackage{amsmath}
\usepackage{amsfonts}
\usepackage{amssymb}
\usepackage{color}
\usepackage[all]{xy}  
\usepackage{enumerate}
\usepackage[top=1in, bottom=1in, left=1in, right=1in]{geometry}
\usepackage{mathrsfs}

\usepackage{stmaryrd}

\theoremstyle{definition}
\newtheorem{theorem}{Theorem}[section]

\newtheorem{corollary}[theorem]{Corollary}
\newtheorem{lemma}[theorem]{Lemma}
\newtheorem{proposition}[theorem]{Proposition}
\newtheorem{claim}[theorem]{Claim}
\newtheorem{subclaim}[theorem]{Subclaim}

\newtheorem{algorithm}[theorem]{Algorithm}
\newtheorem*{claim*}{Claim}

\theoremstyle{definition}
\newtheorem{definition}[theorem]{Definition}

\newtheorem{example}[theorem]{Example}

\newtheorem{remark}[theorem]{Remark}
\numberwithin{equation}{subsection}

\newcommand{\m}{\mathfrak{m}}

\newcommand{\NN}{\mathbb{N}}
\newcommand{\ZZ}{\mathbb{Z}}

\newcommand{\FF}{\mathbb{F}}
\newcommand{\PP}{\mathbb{P}}

\newcommand{\End}{\operatorname{End}}

\newcommand{\Char}{\operatorname{char}}

\newcommand{\CC}{\mathcal{C}}
\newcommand{\supp}{\operatorname{supp}}

\renewcommand{\!}[1]{{\color{red}\text{\Huge$\star$ }#1}}

\newcommand{\ls}{\leqslant}
\newcommand{\gs}{\geqslant}
\newcommand{\ds}{\displaystyle}

\newcommand{\DD}{\mathcal{D}_R}
\def\norm#1{\lvert\lvert #1\rvert\rvert}

\usepackage{fancyhdr}
\usepackage{setspace}

\begin{document}

\author[A.\,F.\,Boix]{Alberto F.\,Boix$^{*}$}
\thanks{$^{*}$Partially supported by MTM2013-40775-P}
\address{Department of Economics and Business, Universitat Pompeu Fabra, Jaume I Building, Ramon Trias Fargas 25-27, 08005 Barcelona, Spain.}
\email{alberto.fernandezb@upf.edu}
\urladdr{http://atlas.mat.ub.edu/personals/aboix/}

\author[A.\,De Stefani]{Alessandro De Stefani$^{\dagger}$}
\thanks{$^{\dagger}$Partially supported by NSF Grant DMS-$1259142$}
\address{Department of Mathematics, University of Virginia, 141 Cabell Drive, Kerchof Hall Charlottesville, VA 22903, USA}
\email{ad9fa@viginia.edu}
\urladdr{http://people.virginia.edu/~ad9fa/}

\author[D.\,Vanzo]{Davide Vanzo}
\address{Departimento di Matematica e Informatica, Universit{\'a} di Firenze, Viale Morgagni, 67/a - 50134 Firenze, Italy}
\email{davide.vanzo@unifi.it}



\title[An algorithm for differential operators in positive characteristic]{An algorithm for constructing certain differential operators in positive characteristic}

\begin{abstract} Given a non-zero polynomial $f$ in a polynomial ring $R$ with coefficients in a finite field of prime characteristic $p$, we present an algorithm to compute a differential operator $\delta$ which raises $1/f$ to its $p$th power. For some specific families of polynomials, we also study the level of such a differential operator $\delta$, i.e., the least integer $e$ such that $\delta$ is $R^{p^e}$-linear. In particular, we obtain a characterization of supersingular elliptic curves in terms of the level of the associated differential operator.
\end{abstract}
\maketitle
\section{Introduction}
Let $R=k[x_1,\ldots,x_d]$ be the polynomial ring over a field $k$, and let $\DD$ be the ring of $k$-linear differential operators on $R$. For every non-zero $f \in R$, the natural action of $\DD$ on $R$ extends uniquely to an action on $R_f$. In characteristic 0, it has been proven by Bernstein in the polynomial ring case (cf.\,\cite[Corollary 1.4]{Bernstein1972}) that $R_f$ has finite length as a $\DD$-module. The minimal $m$ such that $R_f = \DD \cdot \frac{1}{f^m}$ is related to Bernstein-Sato polynomials (cf. \cite[Theorem 23.7, Definition 23.8, and Corollary 23.9]{TwentyFourHours}), and there are examples in which $m >1$ (e.g.\,\cite[Example 23.13]{TwentyFourHours}). Remarkably, in positive characteristic, not only is $R_f$ finitely generated as a $\DD$-module \cite[Proposition 3.3]{Bog}, but it is generated by $\frac{1}{f}$ (cf.\,\cite[Theorem 3.7 and Corollary 3.8]{AMBL}). This is shown by proving the existence of a differential operator $\delta \in \DD$ such that $\delta\left(1/f\right) = 1/f^p$, i.e., a differential operator that acts as the Frobenius homomorphism on $1/f$. The main result of this paper exhibits an algorithm that, given $f\in R$, produces a differential operator $\delta \in \DD$ such that $\delta\left(1/f\right) = 1/f^p$. We will call such a $\delta$ a {\it differential operator associated with $f$}. Our method is described in full details in Section \ref{the algorithm described in detail}. Moreover, this procedure has been implemented using the computer algebra system Macaulay2.

Assume that $\Char(k) = p >0$ and that $[k:k^p]<\infty$. For $e \gs 1$ let $R^{p^e}$ be the subring of $R$ consisting of all $p^e$-th powers of elements in $R$, which can also be viewed as the image of the $e$-th iteration of the Frobenius endomorphism $F:R \to R$. We set $R^{p^0} := R$. It is shown in \cite[1.4.9]{Yek} that $\DD$ is equal to the increasing union $ \bigcup_{e \gs 0} \End_{R^{p^e}}(R)$. Therefore, given $\delta \in \DD$, there exists $e \gs 0$ such that $\delta \in \End_{R^{p^e}}(R)$ but $\delta \notin \End_{R^{p^{e'}}}(R)$ for any $e' < e$. Given a non-zero polynomial $f \in R$, we have seen above that there exists $\delta \in \DD$ that is associated with $f$. We say that $f$ has level $e$ if such $\delta$ is $R^{p^e}$-linear, and there is no $R^{p^{e'}}$-linear differential operator $\delta'$, with $e'<e$, that is associated with $f$.

In Section \ref{monomial section}, we study the case when $f$ is a monomial; indeed, in Theorem \ref{monomial} we determine the level of $f$, and we give an explicit description of the differential operator $\delta$ associated with $f$. We also describe explicitly $I_e(f^{p^e-1})$, the ideal of $p^e$-th roots of $f^{p^e-1}$, where $e$ is the level of $f$. The ideal $I_e(f^{p^e-1})$ can be defined as the unique smallest ideal $J \subseteq R$ such that $f^{p^e-1} \in J^{[p^e]} = (j^{p^e} \mid j \in J)$ (see for example\, \cite[Definition 2.2]{BMS}). In Section \ref{section of level one} we present some families of polynomials which have level one, and we give some examples. In Section \ref{elliptic curves section} we focus on Elliptic Curves $\CC \subseteq \PP^2_{\FF_p}$, where $\FF_p$ is the finite field with $p$ elements. We prove the following characterization: 
\begin{theorem} Let $p \in \ZZ$ be a prime number and let $\CC \subseteq \PP^2_{\FF_p}$ be an elliptic curve defined by a cubic $f(x,y,z) \in \FF_p[x,y,z]$. Then
\begin{enumerate}[(i)]
  \setlength{\itemsep}{2pt}
  \setlength{\parskip}{0pt}
  \setlength{\parsep}{0pt}
\item  $\CC$ is ordinary if and only if $f$ has level one.
\item  $\CC$ is supersingular if and only if $f$ has level two.
\end{enumerate}
\end{theorem}
All computations in this article are made using the computer software Macaulay2 \cite{M2}.

\section{Preliminaries}
The goal of this section is to review the definitions, notations and results that we use throughout this paper. Unless otherwise specified, $k$ will denote a perfect field of prime characteristic $p$. Under this assumption, it is known (see \cite[IV, Th\'eor\`eme 16.11.2]{EGA_4_4}) that the ring of $k$-linear differential operators over $R=k[x_1,\ldots,x_d]$ can be expressed in the following way:
\[
\ds \DD:=R \left\langle D_{x_i,t} \mid i=1,\ldots,d \mbox{ and } t \gs 1 \right\rangle, \ \ \mbox{ where }  D_{x_i,t} := \frac{1}{t!} \frac{\partial^t}{\partial x_i^t}.
\]
This allows us to regard $\DD$ as a filtered ring. Indeed, one has that
\[
\DD=\bigcup_{e\gs 0}\DD^{(e)},\text{ where }\DD^{(e)}:=R \left\langle D_{x_i,t} \mid i=1,\ldots,d \mbox{ and } 1\ls t\ls p^e-1 \right\rangle .
\]
Moreover, it is shown by A.\,Yekutieli (see \cite[1.4.9]{Yek}) that $\DD^{(e)}=\End_{R^{p^e}} (R)$, hence the previous filtration does not depend on the choice of coordinates.

Now, we fix additional notation; given an $\mathbf{\alpha}=(a_1,\ldots ,a_d)\in\NN^d$ we shall use the following multi-index notation:
\[
\mathbf{x}^{\mathbf{\alpha}}:=x_1^{a_1}\cdots x_d^{a_d}.
\]
With this notation, we set $\norm{\mathbf{x}^{\mathbf{\alpha}}}:=\max\{a_1,\ldots ,a_d\}$. By abuse of notation, we will sometimes also use $\norm{\mathbf{\alpha}}$ instead of $\norm{\mathbf{x}^{\mathbf{\alpha}}}$. For any polynomial $g\in k [x_1,\ldots ,x_d]$, we define
\[
\norm{g}:=\max_{\mathbf{x}^{\mathbf{\alpha}}\in\supp (g)} \norm{\mathbf{x}^{\mathbf{\alpha}}},
\]
where if $g=\sum_{\mathbf{\alpha}\in\NN^d} g_{\mathbf{\alpha}} \mathbf{x}^{\mathbf{\alpha}}$ (such that $g_{\mathbf{\alpha}}=0$ for all but a finite number of terms) the support of $g$ is defined as
\[
\supp (g):=\left\{x^{\mathbf{\alpha}}\in R\mid\ g_{\mathbf{\alpha}}\neq 0\right\}.
\]
Moreover, we also define $\deg (g)$ as the total degree of $g$. Finally, for any ideal $J \subseteq R$, $J^{[p^e]}$ will denote the ideal generated by all the $p^e$-th powers of elements in $J$, or equivalently the ideal generated by the $p^e$-th powers of any set of generators of $J$.

\subsection{The ideal of $p^e$-th roots}
Due to the central role which the ideal of $p^e$-th roots plays throughout this article, we review some well-known definitions and facts (cf.\,\cite[page 465]{AMBL} and \cite[Definition 2.2]{BMS}).

\begin{definition}
Given $g\in R$, we set $I_e (g)$ to be the smallest ideal $J\subseteq R$ such that $g\in J^{[p^e]}$.
\end{definition}
\begin{remark} Assume that $k$ is perfect. In our assumptions, the ring $R$ is a free $R^{p^e}$-module, with basis given by the monomials $\{\mathbf{x}^{\mathbf{\alpha}} \mid \norm{\alpha} \ls p^e-1\}$. If we write
\[
g=\sum_{0\ls\norm{\mathbf{\alpha}}\ls p^{e}-1}g_{\mathbf{\alpha}}^{p^e} \mathbf{x}^{\alpha},
\]
then $I_e (g)$ is the ideal of $R$ generated by all the elements $g_{\mathbf{\alpha}}$ \cite[Proposition 2.5]{BMS}.
\end{remark}

\begin{remark} \label{I_e homog} Notice that, if $g$ is a homogeneous polynomial, then, for all $e \in \NN$, $I_e(g)$ is a homogeneous ideal. Indeed, if we write $g=\sum_{0\ls\norm{\mathbf{\alpha}}\ls p^{e}-1}g_{\mathbf{\alpha}}^{p^e} \mathbf{x}^{\alpha}$, then we can assume without loss of generality that every $g_\alpha^{p^e} \mathbf{x}^{\alpha}$ has degree equal to $\deg(g)$. But then $g_\alpha$ must be homogeneous of degree
\[
\ds \deg(g_\alpha) = \frac{\deg(g)-\deg(\mathbf{x}^{\alpha})}{p^e}.
\]
Since $I_e(g)$ is generated by the the elements $g_\alpha$, it is a homogeneous ideal.
\end{remark}

We have the following easy properties (see \cite[Lemma 3.2 and Lemma 3.4]{AMBL} for details).

\begin{proposition}\label{propiedades del ideal raiz}
Given $f\in R$ a non-zero polynomial, and given $e\gs 0$, the following statements hold.

\begin{enumerate}[(i)]

\item $I_e (f)=I_{e+1} (f^p)$.

\item $I_e (f^{p^e-1})\supseteq I_{e+1} (f^{p^{e+1}-1})$.

\end{enumerate}

\end{proposition}
Note that part (ii) of Proposition \ref{propiedades del ideal raiz} produces the following decreasing chain of ideals:
\begin{equation}\label{cadena incordio}
R=I_0 (f^{p^0-1})\supseteq I_1 (f^{p-1})\supseteq I_2(f^{p^2-1})\supseteq I_3 (f^{p^3-1})\supseteq\ldots
\end{equation}
It is shown in \cite{AMBL} that under our assumptions this chain stabilizes. The smallest integer $e \in \NN$ where the chain stabilizes plays a central role in this paper. We summarize the facts that we will need in the following theorem. See \cite[Proposition 3.5, and Theorem 3.7]{AMBL} for details and proofs. 

\begin{theorem}\label{first step of the algorithm}
Let $k$ be a perfect field of prime characteristic $p$, let $R=k[x_1,\ldots ,x_d]$, and let $f\in R \smallsetminus \{0\}$. Define
\[
\ds e:=\inf\left\{s\gs 1\mid\ I_{s-1}\left(f^{p^{s-1}-1}\right)=I_s\left(f^{p^s-1}\right)\right\}.
\]
Then, the following assertions hold.
\begin{enumerate}[(i)]

\item The chain of ideals \eqref{cadena incordio} stabilizes rigidly, that is $e < \infty$ and $I_{e-1}\left(f^{p^{e-1}-1}\right)=I_{e+s} \left(f^{p^{e+s}-1}\right)$ for any $s\gs 0$.

\item One has
\[
e=\min\left\{s\gs 1\mid\ f^{p^s-p}\in I_s\left(f^{p^s-1}\right)^{[p^s]}\right\},
\]
and $e\ls\deg(f)$.

\item There exists $\delta\in\DD^{(e)}$ such that $\delta(f^{p^e-1}) = f^{p^e-p}$, or equivalently such that $\delta (1/f)=1/f^p$.

\item There is no $\delta '\in\DD^{(e')}$, with $e'<e$, such that $\delta ' (1/f)=1/f^p$.
\end{enumerate}
\end{theorem}

Motivated by Theorem \ref{first step of the algorithm}, we make the following definition.
\begin{definition} For a non-zero polynomial $f \in R$, we call the integer $e$ defined in Theorem \ref{first step of the algorithm} the {\it level of $f$}. Also, we will say that $\delta \in \DD^{(e)}$ such that $\delta(f^{p^e-1}) = f^{p^e-p}$, or equivalently such that $\delta(1/f) = 1/f^p$, is a differential operator {\it associated with $f$}.
\end{definition}

\section{The algorithm}\label{the algorithm described in detail}
Let $k$ be a computable perfect field of prime characteristic $p$ (e.g., $k$ is finite). Let $R=k[x_1,\ldots ,x_d]$, and let $f\in R$ be a non-zero polynomial. We now describe in details the algorithm that computes a differential operator $\delta \in \DD$ associated with $f$.
\begin{itemize}
\item {\bf \underline{Step 1}.} Find the smallest integer $e \in \NN$ such that $I_e(f^{p^e-p}) =  I_e(f^{p^e-1})$. There is an implemented algorithm for the computation of the level of a given polynomial $f \in R$. Here follows a description:
\begin{algorithm}\label{calculo del nivel}
Let $k$ be a computable perfect field of prime characteristic $p$ (e.g., finite), let $R=k[x_1,\ldots ,x_d]$, and let $f\in R$. These data act as the input of the procedure. Initialize $e=0$ and $flag=true$. While $flag$ has the value $true$, execute the following commands:

\begin{enumerate}[(i)]

\item Assign to $e$ the value $e+1$, and to $q$ the value $p^e$.

\item Compute $I_e (f^{q-1})$.

\item Assign to $J$ the value $I_e \left(f^{q-1}\right)^{[q]}$. 

\item If $f^{q-p}\in J$, then $flag=false$; otherwise, come back to step (i).

\end{enumerate}
At the end of this method, return the pair $\left(e,I_e\left(f^{p^e-1}\right)\right)$. Such $e$ is exactly the $e$ described in Theorem \ref{first step of the algorithm}, i.e., the level of $f$.
\end{algorithm}

\begin{remark}
Since the level $e$ is always at most $\deg(f)$, we can ensure that the While loop in Algorithm \ref{calculo del nivel} finishes after, at most, $\deg(f)$ iterations. Notice that, a priori, there is a black box in this algorithm, namely the computation of $I_e (f^{q-1})$ at each step. However, the calculation of the ideal of $p^e$-th roots is well known (cf. \cite[Section 6]{KatzmanSchwede2012}).
\end{remark}

\begin{remark}
As pointed out by E.\,Canton in \cite[Definition 2.3]{Canton15}, the so-called \emph{non-F-pure ideal} of $f$ introduced by O.\,Fujino, K.\,Schwede and S.\,Takagi in \cite[Definition 14.4]{FujinoSchwedeTakagi11} turns out to be $I_e \left(f^{p^e-1}\right)$, where $e$ is the level of $f$ (see \cite[Remark 16.2]{FujinoSchwedeTakagi11}). Therefore, Algorithm \ref{calculo del nivel} provides a procedure to calculate the non-F-pure ideal.
\end{remark}

\item {\bf \underline{Step 2}.} For $e \in \NN$ as in {\bf Step 1} write $f^{p^e-1} = \sum_{i=1}^n c_i^{p^e} \mu_i$, where $\{\mu_1,\ldots,\mu_n\}$ is the basis of $R$ as an $R^{p^e}$-module consisting of all the monomials $x_1^{a_1}\cdots x_d^{a_d}$, with $a_i\ls p^e-1$ for all $i=1,\ldots,d$. 
\begin{claim}\label{Dirac delta claim} For all $i =1,\ldots,n$ there exists $\delta_i \in \DD^{(e)}$ such that $\delta_i(\mu_j) = 1$ if $i=j$ and $\delta_i(\mu_j) = 0$ if $i\ne j$.
\end{claim}
\begin{proof} For $i \in \{1,\ldots,n\}$ and $\mu_i = x_1^{a_1}\cdots x_d^{a_d}$ consider $\nu_i:=x_1^{p^e-1-a_1}\cdots x_d^{p^e-1-a_d}$, which is a monomial in $R$ because $a_k \ls p^e-1$ for all $k=1,\ldots,d$. Notice that $\nu_i \mu_j = (x_1\cdots x_d)^{p^e-1}$ if and only if $i=j$. Then set 
\[
\ds \delta_i := \left( \prod_{k=1}^d D_{x_k,p^e-1} \right) \cdot \nu_i \in \DD^{(e)}.
\]
Notice that $\delta_i(\mu_i) = 1$, and that if $\mu_j=x_1^{b_1}\cdots x_d^{b_d}$, then $\delta_i(\mu_j) = 0$ if $b_k < a_k$ for some $k\in\{1,\ldots,d\}$. So let us assume that $a_k \ls b_k \ls p^e-1$ for all $k=1,\ldots,d$, and that there is $s\in \{1,\ldots,d\}$ such that $a_s<b_s$. Note that by definition of $\nu_i$ we have that $\nu_i\mu_j = x_1^{r_1}\cdots x_d^{r_d}$, with $p^e \ls r_s \ls 2p^e-2$, so that we can write $r_s = p^e+n$ for some integer $n$ with $0 \ls n \ls p^e-2$. 
\begin{subclaim}The coefficient of $D_{x_s,p^e-1}(\nu_i\mu_j)$ is  
\[
{p^e+n \choose p^e-1}  \equiv 0  \mod p.
\]
\end{subclaim}
\begin{proof} As a consequence of a theorem proved by Lucas in \cite[pp.\,51--52]{Lucas}, it is enough to check that at least one of the digits of the base $p$ expansion of $p^e-1$ is greater than the corresponding digit in the base $p$ expansion of $p^e+n$. The base $p$ expansion of $p^e-1$ is given by
\[
p^e-1 = (p-1) (1+p+\cdots+ p^{e-1}) = (p-1)p^0 + (p-1)p^1 + \cdots + (p-1)p^{e-1},
\]
so that the subclaim is proved unless the first $e$ digits of $p^e+n$ are $p-1$ as well. But in this case, since $p^e+n>p^e-1$ we would get
\[
\ds p^e+n \gs (p-1)p^0 + (p-1)p^1 + \cdots + (p-1)p^{e-1} + p^e = 2p^e-1,
\]
a contradiction since $n \ls p^e-2$.
\end{proof}
The Subclaim shows that $D_{x_s,p^e-1}(\nu_i\mu_j) = 0$ for all $\mu_j$ with $j \ne i$. Therefore, using that $\delta_i \in \DD^{(e)} = \End_{R^{p^e}}(R)$, we get
\[
\ds \delta_i(f^{p^e-1}) = \delta_i\left(\sum_{j=1}^n c_j^{p^e} \mu_j\right)=\sum_{j=1}^n c_j^{p^e} \delta_i(\mu_j) = c_i^{p^e}.\qedhere
\]
\end{proof}
\item {\bf \underline{Step 3}.} Since $1 \in \DD^{(e)}$, for $e \in \NN$ as in {\bf Step 1} we have
\[
f^{p^e-p} \in \DD^{(e)}(f^{p^e-p}) = I_e(f^{p^e-p})^{[p^e]} = I_e(f^{p^e-1})^{[p^e]} = (c_1,\ldots,c_n)^{[p^e]}.
\]
In particular there exist $\alpha_1,\ldots,\alpha_n \in R$ such that $f^{p^e-p} = \sum_{i=1}^n \alpha_i c_i^{p^e}$. Consider $\delta_i \in \DD^{(e)}$ as in {\bf Step 2}, so that $\delta_i(f^{p^e-1}) = c_i^{p^e}$, and set
\[
\ds \delta:= \sum_{i=1}^n \alpha_i \delta_i \in \DD^{(e)}.
\]
With this choice we have
\[
\ds \delta(f^{p^e-1}) = \delta\left(\sum_{j=1}^n c_j^{p^e} \mu_j\right) = \sum_{i,j=1}^n c_j^{p^e}\alpha_i \delta_i(\mu_j) = \sum_{i=1}^n \alpha_i c_i^{p^e} = f^{p^e-p},
\]
and using that $\delta \in \DD^{(e)}$ we finally get
\[
\ds \delta\left(\frac{1}{f}\right) = \delta \left(\frac{f^{p^e-1}}{f^{p^e}}\right) = \frac{1}{f^{p^e}} \delta \left(f^{p^e-1}\right) = \frac{f^{p^e-p}}{f^{p^e}} = \frac{1}{f^p}.
\]
\end{itemize}
\section{The monomial case}\label{monomial section} 
Throughout this section, let $k$ be a perfect field and let $R=k[x_1,\ldots,x_d]$. We now analyze the case when $f \in R$ is a monomial. First we show a lower bound for the level of $f$.
\begin{lemma}\label{lower bound of the level}Let $f=x_1^{a_1}\cdots x_d^{a_d}$ be a monomial in $R=k[x_1,\ldots,x_d]$, with $a_i >0$ for all $i=1,\ldots,d$. Let $\delta \in \DD^{(e)}$ be such that $ \delta\left(1/f\right) = 1/f^p$. Then, setting $a:=\norm{f} = \max \{a_i \mid 1 \ls i \ls d\}$, we have
\[
\ds e \gs \left\lceil \log_p(a) \right\rceil +1.
\]
\end{lemma}
\begin{proof} It suffices to show that for $t :=  \left\lceil \log_p(a)\right\rceil$ we have $I_{t}(f^{p^{t}-p}) \supsetneq I_{t}(f^{p^{t}-1})$. This is because if the chain 
\[
\ds I_t(f^{p^t-p^{t-1}}) \supseteq I_t(f^{p^t-p^{t-2}}) \supseteq \ldots \supseteq I_t(f^{p^t-p}) \supsetneq I_t(f^{p^t-1})
\]
stabilizes before such step, then it would be stable at it as well, and the smallest $s$ such that $I_s(f^{p^s-p}) = I_s(f^{p^s-1})$ is precisely the level of $f$ (see Theorem \ref{first step of the algorithm}). Notice that $t=0$ if and only if $a_i=1$ for all $i$. For such a monomial the Lemma is trivially true. So let us assume that $t \gs 1$, or equivalently that $a_i \gs 2$ for at least one $i \in \{1,\ldots, d\}$. Let 
\[
\ds j_i:= \min\{j \in \NN \mid jp^t \gs  a_i\},
\]
and notice that $j_i \ls a_i$ for all $i$, and by choice of $t$ we have that $j_i \gs 2$ for at least one $i$, say $j_1 \gs 2$. Then
\[
f^{p^t-p} = x_1^{p^ta_1-pa_1} \cdots x_d^{p^ta_d - pa_d} = (x_1^{a_1-j_1}\cdots x_d^{a_d-j_d})^{p^t} \cdot x_1^{j_1p^t-pa_1}\cdots  x_d^{j_dp^t-pa_d}.
\]
Since $(j_i-1)p^t<a_i$ by definition of $j_i$, we have $j_ip^t-pa_i < p^t+a_i-pa_i < p^t$. This shows that $I_t(f^{p^t-p}) = (x_1^{a_1-j_1}\cdots x_d^{a_d-j_d})$. On the other hand:
\[
f^{p^t-1} = x_1^{p^ta_1-a_1} \cdots x_d^{p^ta_d - a_d} = (x_1^{a_1-1}\cdots x_d^{a_d-1})^{p^t} \cdot x_1^{p^t-a_1} \cdots x_d^{p^t-a_d},
\]
which makes sense because $p^t\gs a_i$ for all $i$. This shows that $I_t(f^{p^t-1}) =  (x_1^{a_1-1}\cdots x_d^{a_d-1})$, and because $j_1 \gs 2$ we have
\[
I_t(f^{p^t-p})= (x_1^{a_1-j_1}\cdots x_d^{a_d-j_d}) \supseteq (x_1^{a_1-2}\cdots x_d^{a_d-1}) \supsetneq (x_1^{a_1-1}\cdots x_d^{a_d-1}) = I_t(f^{p^t-1}).\qedhere
\]
\end{proof}
\begin{theorem} \label{monomial} Let $f=x_1^{a_1}\cdots x_d^{a_d}$ be a monomial in $k[x_1,\ldots,x_d]$, with $a_i >0$ for all $i=1,\ldots,d$. Let $a:= \norm{f} = \max \{a_i \mid 1 \ls i \ls d\}$. Then $f$ has level $e := \left\lceil \log_p(a)  \right\rceil+1$, and $I_e(f^{p^e-1}) = (x_1^{a_1-1}\cdots x_d^{a_d-1})$. Furthermore,
\[
\ds \delta:=  \prod_{i=1}^d \left(x_i^{p^e-pa_i} \cdot D_{x_i,p^e-1} \cdot x_i^{a_i-1} \right) \in \DD^{(e)}
\]
is a differential operator associated with $f$.
\end{theorem}
\begin{proof} Set $e := \left\lceil \log_p(a)  \right\rceil+1$. During the proof of Lemma \ref{lower bound of the level} we already proved that $I_e(f^{p^e-1}) = (x_1^{a_1-1}\cdots x_d^{a_d-1})$; keeping this fact in mind, it is enough to check that $\delta(f^{p^e-1}) = f^{p^e-p}$. Indeed, we have
\begin{eqnarray*}
\begin{array}{c}
\ds \delta(f^{p^e-1}) = \delta\left(x_1^{p^ea_1-a_1} \cdots x_d^{p^ea_d-a_d}\right) = \left(x_1^{a_1-1} \cdots x_d^{a_d-1}\right)^{p^e} \delta\left(x_1^{p^e-a_1} \cdots x_d^{p^e-a_d}\right) = \\ \\
\ds = \prod_{i=1}^d \left(x_i^{p^ea_i-pa_i}\cdot D_{x_i,p^e-1}(x_i^{p^e-1})\right) = x_1^{p^ea_1-pa_1} \cdots x_d^{p^ea_d-pa_d} = f^{p^e-p},
\end{array}
\end{eqnarray*}
and therefore the proof is completed.
\end{proof}

Regarding the level $e$ obtained in Theorem \ref{monomial}, one might ask whether, given any non-zero $f\in R$, its level would always be bounded above by $\lceil\log_p (\norm{f})\rceil +1$. Unfortunately, this is not the case, as the following example illustrates.

\begin{example}
Consider $f:=xy^3+x^3\in\FF_2[x,y]$. In this case, one can check with Macaulay2 that the level of $f$ is $4$, while $\lceil\log_2 (\norm{f})\rceil +1=3$. In fact, the level is even strictly greater than $\lceil\log_2(\deg(f))\rceil + 1=3$.
\end{example}

The monomial in Theorem \ref{monomial} is assumed to be of the form $x_1^{a_1}\cdots x_d^{a_d}$. Using a suitable linear change of coordinates, we immediately get the following Corollary, which includes the general monomial case.
\begin{corollary} \label{linforms} Let $n\ls d$, let $f=\ell_1^{a_1}\cdots \ell_n^{a_n}$ be a product of powers of  linear forms which are linearly independent over $k$, and let $a:=\max \{a_i \mid 1\ls i \ls n\}$. Then $f$ has level $e = \left\lceil \log_p(a)  \right\rceil+1$, the ideal of $p^e$-th roots is $I_e(f^{p^e-1}) = (\ell_1^{a_1-1}\cdots \ell_n^{a_n-1})$ and
\[
\ds \delta:=  \prod_{i=1}^n \left(\ell_i^{p^e-pa_i} \cdot D_{\ell_i,p^e-1} \cdot \ell_i^{a_i-1} \right) \in \DD^{(e)}
\]
is a differential operator associated with $f$. Here, if $\ell_i = \sum_{j=1}^d \lambda_{ij} x_j$, then the differential operator $D_{\ell_i, p^e-1}$ is defined as $\sum_{j=1}^d \lambda_{ij} D_{x_j,p^e-1}$.
\end{corollary}

\section{Families of level one}\label{section of level one}
Polynomials of level one, that is polynomials $f$ such that $I_1(f^{p-1}) = R$, are somehow special. For instance, let $f,g \in R$ and let $\delta \in \DD$ be associated with $f$. Assume that $e = 1$, then for $\delta':=\delta\left( g^{p-1} \cdot \underline{ \ \ }\right)$ we get
\[
\ds \delta'\left(\frac{g}{f}\right) = \delta\left(\frac{g^p}{f}\right) = g^p \cdot \delta \left(\frac{1}{f}\right) = \left(\frac{g}{f}\right)^p.
\]
The authors do not know whether, for any choice of $f,g \in R$, $f \ne 0$, there always exists $\delta' \in \DD$ such that $\delta'(g/f) = (g/f)^p$. In fact, when $\delta \in \DD^{(e)}$ for $e \gs 2$, the best we can get is $\delta'\left(g/f\right) = g^{p^e}/f^{p}$, with $\delta':= \delta\left(g^{p^e-1} \cdot \underline{ \ \ } \right)$. On the other hand, for any $f \in R$ we have $R_f \cong \DD \cdot \frac{1}{f}$ and therefore, for any $g \in R$, there exists $\delta' \in \DD$ such that $\delta'\left(1/f\right) = g^p/f^p$. In fact it is enough to set $\delta':= g^{p} \cdot \delta$.





We will now exhibit some families of polynomials that have level one, together with some examples. However, before doing so, we want to single out the following elementary statement, because we will be using it repeatedly throughout this section. It may be regarded as a straightforward sufficient condition which ensures that a polynomial has level one. In this section, unless otherwise stated, $k$ will denote a perfect field and $R=k[x_1,\ldots,x_d]$ will be a polynomial ring over $k$.

\begin{lemma}\label{sqfpowers_nonhomog}
Let $f \in R$ be a non-zero polynomial, write
\[
f^{p^e-1}=\sum_{0\ls\norm{\mathbf{\alpha}}\ls p^e-1}f_{\mathbf{\alpha}}^{p^e} \ \mathbf{x}^{\mathbf{\alpha}},
\]
and assume that $f_{\mathbf{\beta}}$ is a unit for some $0\ls\norm{\mathbf{\beta}}\ls p^e-1$ and some $e\gs 1$. Then, $f$ has level one.
\end{lemma}

\begin{proof}
By definition, we have that $I_e(f^{p^e-1}) = R$; on the other hand, we know that $I_1 (f^{p-1})\supseteq I_e(f^{p^e-1})$. In this way, combining these two facts it follows that $I_1(f^{p-1}) = R$, and therefore $f$ has level one.
\end{proof}
We can give an easy but useful characterization of homogeneous polynomials of level one.
\begin{lemma} \label{sqfpowers} Let $f \in R$ be a homogeneous non-zero polynomial. Let $\{\mu_j\}_{j=1}^{p^d}:= \{\mathbf{x}^{\mathbf{\alpha}} \mid \norm{\alpha} \ls p-1\}$ be the monomial basis of $R$ as a $R^p$-module. Then $f$ has level one if and only if $\mu_j \in \supp(f^{p-1})$ for some $j = 1,\ldots,p^d$.
\end{lemma}
\begin{proof} Write 
\[
f^{p-1}=\sum_{0\ls\norm{\mathbf{\alpha}}\ls p-1}f_{\mathbf{\alpha}}^p \  \mathbf{x}^{\mathbf{\alpha}}.
\]
Note that, since $f$ is homogeneous, $I_1(f^{p-1})$ is a homogeneous ideal by Remark \ref{I_e homog}. If $f$ has level one, then $I_1(f^{p-1}) = (f_\alpha  \mid 0 \ls \norm{\alpha} \ls p-1) = R$, therefore there exists at least one coefficient $f_\beta$ that is outside of the irrelevant maximal ideal $\m=(x_1,\ldots, x_d)$. Write $f_\beta = \lambda + r$, with $\lambda \in k$ and $r \in \m$. Then, we can write  $f^{p-1} = \lambda^p \mathbf{x}^{\mathbf{\beta}} + h$ for some $h \in R$. Also, since $\{\mathbf{x}^{\mathbf{\alpha}}  \mid 0 \ls \norm{\alpha} \ls p-1\}$ is a basis of $R$ as $R^p$-module, there is no cancellation between $\lambda^p \mathbf{x}^{\mathbf{\beta}}$ and $h$. Thus, $\mu_j=\mathbf{x}^{\mathbf{\beta}} \in \supp(f^{p-1})$. Conversely, if $\mu_j = \mathbf{x}^{\mathbf{\beta}} \in \supp(f^{p-1})$ for some $0 \ls \norm{\beta} \ls p-1$, then we can write again $f^{p-1} = \lambda^p \mathbf{x}^{\mathbf{\beta}} + h$ for some $\lambda \in k$ and some $h \in R$. Also, we can assume that $\mathbf{x}^{\mathbf{\beta}} \notin \supp(h)$. Then the coefficient $f_\beta^p$ of $\mathbf{x}^{\mathbf{\beta}}$ in the expansion of $f^{p-1}$ must be $\lambda^p+r^p = (\lambda+r)^p$ for some $r \in \m$, and thus $\lambda+r \in I_1(f^{p-1})$. Since the latter is homogeneous (here we are using again Remark \ref{I_e homog}), we have in particular that $\lambda \in I_1(f^{p-1})$, which implies $I_1(f^{p-1}) = R$, and therefore $f$ has level one.
\end{proof}


\begin{proposition} \label{sqfree} Let $f \in R$ be a non-zero polynomial whose support contains a squarefree term involving a variable that does not appear in any other term of the support of $f$. Then $f$ has level one.
\end{proposition}
\begin{proof} Without loss of generality we can assume that $x_{1}\cdots x_{n} \in \supp(f)$, and that $x_1$ does not appear in any other term of $\supp(f)$. Write $f=\lambda_0 x_1 \cdots x_n+ \sum_{i=1}^s \lambda_i m_i$, where $m_i=x^{\alpha_i}$ and $\alpha_i \in \NN^d$ are of the form $(0,\alpha_{i2},\ldots,\alpha_{id})$ for $i=1,\ldots,s$. Then,
\[
f^{p-1} = \sum_{i_0+i_1+\cdots+i_s = p-1} {p-1 \choose i_0, \ldots, i_s} \lambda_0^{i_0}\lambda_1^{i_1} \cdots \lambda_s^{i_s} (x_1\cdots x_n)^{i_0} m_1^{i_1}\cdots m_s^{i_s}.
\]
Notice that in order for a term in the support of $f^{p-1}$ to be divisible by $x_1^{p-1}$ it is necessary that $i_0 = p-1$, in which case $i_1 = \ldots = i_s = 0$. Hence $x_1^{p-1} \cdots x_n^{p-1}$ is in the support of $f^{p-1}$, appearing with coefficient $\lambda_0^{p-1} \ne 0$. Then we can write
\[
\ds f^{p-1} = \lambda_0^{p-1} x_1^{p-1} \cdots x_n^{p-1} + \sum_{{\tiny \begin{array}{c} 0 \ls \norm{\alpha} \ls p-1 \\ \alpha_1\ne p-1\end{array}}} f_\alpha^p  \ \mathbf{x}^\alpha,
\]
and the Proposition now follows from Lemma \ref{sqfpowers_nonhomog}, since $\lambda_0^{p-1}$ is a unit.
\end{proof}
\begin{example} Let $f=x^2+y^2+xyz \in \FF_p[x,y,z]$. Since $z$ appears in the square free term $xyz$ and nowhere else in the support of $f$ we have that $f$ has level one by Proposition \ref{sqfree}. In fact, $\delta:= D_{x,p-1}D_{y,p-1}D_{z,p-1} \in \DD^{(1)}$ is a differential operator associated with $f$.
\end{example}

\begin{proposition} \label{allsqfree} Let $f \in R$ be a non-zero polynomial of degree $n$ such that every element of its support is a squarefree monomial. Then $f$ has level one.
\end{proposition}
\begin{proof} Without loss of generality we can assume that $x_{1}\cdots x_n \in \supp(f)$. We want to show that we can apply Lemma \ref{sqfpowers_nonhomog}. Let $f= \lambda_0 x_1\cdots x_n+ \sum_{i=1}^s \lambda_i m_i$, where the monomials $m_i$ are squarefree of degrees $d_i:= \deg(m_i) \ls n$, that we can assume different from $x_1\cdots x_n$. Then
\[
f^{p-1} = \sum_{i_0+i_1+\cdots+i_s = p-1} {p-1 \choose i_0, \ldots, i_s} \lambda_0^{i_0}\lambda_1^{i_1} \cdots \lambda_s^{i_s} (x_1\cdots x_n)^{i_0} m_1^{i_1}\cdots m_s^{i_s}.
\]
Note that the choice $i_0=p-1$, $i_1=\ldots=i_s=0$ gives the monomial $\lambda_0^{p-1} (x_1\cdots x_n)^{p-1}$, and we want to show that this choice of indices is the only one that gives such a monomial. By way of contradiction, assume that $(x_1\cdots x_n)^{i_0} m_1^{i_1}\cdots m_s^{i_s} = (x_1\cdots x_n)^{p-1}$, then necessarily each $m_i$ divides $x_1\cdots x_n$, because they are squarefree. Since we are assuming that none of the monomials $m_i$ is equal to $x_1\cdots x_n$, we must have that $\deg(m_i) < n$. But then
\[
\ds \deg\left((x_1\cdots x_n)^{i_0} m_1^{i_1}\cdots m_s^{i_s}\right) = ni_0+d_1i_1+\cdots+d_si_s < n(i_0+i_1+\cdots+i_s) = n(p-1),
\]
which is a contradiction because $(x_1\cdots x_n)^{i_0} m_1^{i_1}\cdots m_s^{i_s} = (x_1\cdots x_n)^{p-1}$, and the degree of the latter is $n(p-1)$. Therefore if we write
\[
\ds f^{p-1} =  \sum_{0 \ls \norm{\alpha} \ls p-1} f_\alpha^p  \ \mathbf{x}^\alpha,
\]
then the coefficient of $x_1^{p-1}\cdots x_n^{p-1}$ is precisely $\lambda_0^{p-1}$, which is a unit. Using Lemma \ref{sqfpowers_nonhomog}, the Proposition now follows.
\end{proof}

\begin{example} Let $R=\FF_p[X_{ij}]_{1\ls i,j \ls n}$ be a polynomial ring in $n^2$ variables and let $f=\det(X_{ij})$. Because of Proposition \ref{allsqfree} $f$ has level one, since its support consists only of squarefree monomials.
\end{example}

\begin{proposition}\label{quadratic forms} Let $f \in R = k[x_1,\ldots,x_d]$ be a homogeneous quadric. Then $f$ has level one unless $f$ is the square of a linear form, in which case $f$ has level two.
\end{proposition}
\begin{proof} If $f$ is a power of a linear form, then $f$ has level two by Corollary \ref{linforms}. Otherwise, if $p \ne 2$ there exists a linear change of variables that diagonalizes $f$ (cf.\,\cite[Chapter IV, Proposition 5]{arithmeticcourse}). That is, we can assume that, after a linear change of coordinates, $f=x_1^2 + \cdots + a x_n^2$, where $2 \ls n \ls d$ and $a$ is either $1$ or an element of $k$ which is not a square. Notice that $x_1^{p-1}x_2^{p-1}$ appears with coefficient $\lambda:= {p-1 \choose \frac{p-1}{2}} \in k \smallsetminus \{0\}$ if $n \gs 3$, and with coefficient $a^{(p-1)/2} {p-1 \choose \frac{p-1}{2}} \in k \smallsetminus\{0\}$ if $n=2$. Therefore $f$ has level one by Lemma \ref{sqfpowers}. Finally, if $p=2$ and $f$ is not a power of a linear form, then we can assume that $x_1x_2$ appears with non-zero coefficient in $f^{p-1} = f$, and we conclude using again Lemma \ref{sqfpowers}.
\end{proof}

\begin{proposition}\label{diagonal hypersurfaces: no todas}
Let $f=x_1^t+\cdots+x_d^t \in R$ be a diagonal hypersurface of degree $t\gs 1$. If $t\ls\min\{d,p\}$ and $p\equiv 1\pmod t$, then $f$ has level one.




\end{proposition}

\begin{proof}

Our assumptions on $t,d$ and $p$ allow us to expand $f^{p-1}$ in the following manner:
\[
\ds f^{p-1}=\frac{(p-1)!}{\left(\left(\frac{p-1}{t}\right)!\right)^t}x_1^{p-1}\cdots x_t^{p-1}+\sum_{\substack{i_1+\cdots+i_d=p-1\\ (i_1,\ldots ,i_t)\neq ((p-1)/t,\ldots,(p-1)/t)}} \frac{(p-1)!}{i_1!\cdots i_d!}x_1^{t i_1}\cdots x_d^{t i_d}.
\]
Since $\frac{(p-1)!}{\left(\left(\frac{p-1}{t}\right)!\right)^t} \in k\smallsetminus\{0\}$, the above equality shows that $x_1^{p-1}\cdots x_t^{p-1}$ appears in $\supp (f^{p-1})$ with non-zero coefficient, hence $f$ has level one by Lemma \ref{sqfpowers}.
\end{proof}

The assumptions $p\equiv 1\pmod t$ and $t\ls\min\{d,p\}$ in Proposition \ref{diagonal hypersurfaces: no todas} cannot be removed in general, as the following examples illustrate.

\begin{example}
Let $R:=\mathbb{F}_5 [x,y,z]$ and $f:=x^3+y^3+z^3$. One can check using Macaulay2 that $f$ has level two. Notice that, in this case $3=\deg (f)\ls\min\{3,5\}$ and $5\equiv 2\pmod 3$.

On the other hand, consider now $R:=\mathbb{F}_7 [x,y]$ and $f:=x^3+y^3$. One can check using Macaulay2 that $f$ has level two. In this case $3=\deg (f)> 2=\min\{2,7\}$ and $p=7\equiv 1\pmod 3$.
\end{example}
The diagonal hypersurface considered in Proposition \ref{diagonal hypersurfaces: no todas} is of the form $x_1^t+\cdots+x_d^t$; using a suitable linear change of coordinates, we immediately get the following Corollary, which includes as a particular case Proposition \ref{diagonal hypersurfaces: no todas}.

\begin{corollary}\label{linforms2}
Let $n\ls d$, let $f=\ell_1^t+\cdots+\ell_n^t$ be a diagonal hypersurface of degree $t\gs 1$ made up by linear forms $\ell_1,\ldots ,\ell_n$ which are linearly independent over the field $k$. If $t\ls\min\{n,p\}$ and $p\equiv 1\pmod t$, then $f$ has level one.
\end{corollary}

Before going on, we want to review the following notion (see \cite[page 243]{TwentyFourHours}):

\begin{definition}
A polynomial $f\in R$ is said to be \emph{regular} provided
\[
\operatorname{Tj}(f):=\left(f,\frac{\partial f}{\partial x_1},\ldots ,\frac{\partial f}{\partial x_d}\right)=R,
\]
where $\operatorname{Tj}(f)$ denotes the Tjurina ideal attached to $f$.
\end{definition}
In characteristic zero, a polynomial is regular if and only if its Bernstein-Sato polynomial is $b_f(s) = s+1$ \cite[Theorem 23.12]{TwentyFourHours}. In this case, $R_f$ is generated by $1/f$ as a $D$-module.
\begin{proposition}\label{some regular polynomials of level one}
Let $k$ be a perfect field of characteristic $2$, and let $f\in k[x_1,\ldots ,x_d]$ be regular. Then, $f$ has level one.
\end{proposition}

\begin{proof}
Since $f$ is regular, there are $r_0,r_1,\ldots ,r_d\in R$ such that
\[
1=r_0 f+r_1 \frac{\partial f}{\partial x_1}+\cdots +r_d \frac{\partial f}{\partial x_d}.
\]
In this way, setting
\[
\delta:= r_0+\sum_{j=1}^d  \frac{\partial}{\partial x_j}
\]
it follows that $\delta (f)=1$ and therefore $f$ has level one.
\end{proof}

\begin{remark}\label{more regular polynomials of level one}
A very easy way to produce polynomials which are simultaneously regular and of level one in arbitrary prime characteristic works as follows.
Let $k$ be a perfect field of prime characteristic $p$, let $R=k[x_1,\ldots ,x_d]$, and assume that $f\in R$ is a non-zero polynomial of the form $f=\lambda x_i +g$, for some $1\ls i\ls d$, some $\lambda\in k -\{0\}$, and some $g\in R$ such that
\[
\frac{\partial g}{\partial x_i}=0\quad (\text{i.e., }g\in k[x_1,\ldots ,x_{i-1},\widehat{x_i},x_{i+1},x_{i+2},\ldots ,x_d]).
\]
Then, $f$ is regular and of level one; indeed, the fact that $f$ is of level one follows directly from Proposition \ref{sqfree}.
\end{remark}


\section{Elliptic Curves}\label{elliptic curves section}
Let $p \in \ZZ$ be a prime and let $\CC \subseteq \PP^2_{\FF_p}$ be an elliptic curve defined by an homogeneous cubic $f(x,y,z) \in \FF_p[x,y,z]$. We want to review here the following notion (see \cite[13.3.1]{Husemoller}).
\begin{definition} $\CC$ is said to be ordinary if the monomial $(xyz)^{p-1}$ appears in the expansion of $f^{p-1}$ with non-zero coefficient. Otherwise, $\CC$ is said to be supersingular.
\end{definition}
The general form of a cubic defining an elliptic curve is the following
\[
f= y^2z+a_1xyz +a_3yz^2-x^3-a_2x^2z-a_4xz^2-a_6z^3,
\]
where $a_1,\ldots,a_6 \in \FF_p$. When $p\ne 2,3$, the expression above can be further simplified to
\[
f= y^2z - x^3 +axz^2 + bz^3,
\]
for $a,b \in \FF_p$ (see \cite[3.3.6]{Husemoller} for details). We are now interested in computing the level of elliptic curves $\CC$. We are mainly interested in upper bounds, since it is easy to see from Lemma \ref{sqfpowers} that any ordinary elliptic curve has level one, and that any supersingular elliptic curve has level at least two. First, we explore the low characteristic cases, where the list of possibilities (up to isomorphism) is very limited.
\begin{proposition} \label{char_2_3} Let $\CC \subseteq \PP^2_{\FF_p}$ be a supersingular elliptic curve defined by a cubic $f \in \FF_p[x,y,z]$. If $p = 2$ or $p = 3$, then $f$ has level two.
\end{proposition}

\begin{proof} Set $D:=D_{x,p^2-1}D_{y,p^2-1}D_{z,p^2-1}$. By \cite[13.3.2 and 13.3.3]{Husemoller} there are only the following two cases, up to isomorphism:
\begin{center}
\begin{tabular}{|c|c|c|}
\hline
$p$ & Elliptic curve & Differential operator \\ 
\hline
$2$ & $x^3+y^2z+yz^2$ & $ 
\begin{array}{c} \\
y^2 D x^3z+ z^2Dx^3y+x^2Dxyz^2 \\ \\
\end{array}$\\ 
\hline
$3$ & $x^3-xz^2-y^2z$ & $
\begin{array}{c} \\ (x^6z^3-x^3y^6)Dx^4z^5+\\
+(x^9+x^3z^6+y^6z^3)Dxy^8+y^3z^6Dx^4y^5 \\ \\
\end{array}$\\  
\hline
\end{tabular}
\end{center}
The table above is exhibiting a differential operator of level two for each polynomial, therefore the level is at most two in all such cases. We have already noticed that $\CC$ is ordinary if and only if $f$ has level one. This shows that when $p=2$ or $p=3$, $\CC$ is supersingular if and only if $f$ has level two.
\end{proof}

Recall that $R = \FF_p[x,y,z]$ is a free $R^{p^2}$-module with basis given by $\{x^ry^sz^t \mid 0 \ls r,s,t \ls p^2-1\}$. For a polynomial $g \in \FF_p[x,y,z]$ consider 
\[
\ds g = \sum_{0 \ls r,s,t \ls p^2-1} (c(r,s,t))^{p^2} \ x^ry^sz^t,
\]
and recall that by Proposition \ref{propiedades del ideal raiz}, with this notation, $I_2(g)$ is the ideal generated by the elements $c(r,s,t)$, for $0 \ls r,s,t \ls p^2-1$.
\begin{remark} If $f \in \FF_p[x,y,z]$ is a cubic and $g = f^{p^2-1}$, then for any $0 \ls r,s,t \ls p^2-1$ one has
\[
\deg(c(r,s,t)) \ls \left \lfloor \frac{\deg(f^{p^2-1})}{p^2}\right \rfloor = \left\lfloor \frac{3(p^2-1)}{p^2} \right\rfloor = 2.
\]
In particular, $I_2(f^{p^2-1}) = \ds \left(c(r,s,t) \mid 0 \ls r,s,t \ls p^2-1\right)$ is generated in degree at most two. 
\end{remark}
For the rest of the section, we will denote $I_1(f^{p-1})$ and $I_2(f^{p^2-1})$ simply by $I_1$ and $I_2$.

\subsection{Preliminary computations}
The purpose of this part is to single out some technical facts which will be used for proving the main result of this section; namely, Theorem \ref{elliptic_curves}.

\begin{lemma} \label{y} Let $f = y^2z-x^3+axz^2+bz^3 \in \FF_p[x,y,z]$, where $p \ne 2,3$. Then 
\[
\ds c(0,p^2-2,p^2-1) = y.
\]
\end{lemma}
\begin{proof} A monomial in the expansion of $f^{p^2-1}$ will have the form $(y^2z)^h(-x^3)^i(axz^2)^j(bz^3)^k$, where $h+i+j+k=p^2-1$. Looking at the coefficient of $y^{p^2-2}z^{p^2-1}$ in such expansion, by degree considerations we only have three possibilities:
\[
\ds y^{2h}x^{3i+j} z^{h+2j+3k} = \left\{
\begin{array}{ll}
x^{p^2} \cdot y^{p^2-2}z^{p^2-1} \\
y^{p^2} \cdot y^{p^2-2}z^{p^2-1} \\
z^{p^2} \cdot y^{p^2-2}z^{p^2-1} 
\end{array}
\right.
\]
Since $p^2-2$ is not even, there is no choice of $h$ that realizes the first and the third cases. So we are left with the second, which is achieved only by the choice $h=p^2-1$, $i=j=k= 0$. This shows that the coefficient of $y^{p^2-2}z^{p^2-1}$ in the expansion of $f^{p^2-1}$ is precisely $y^{p^2}$, and the Lemma follows.
\end{proof}

Before going on, we need to review the following classical result, due to Legendre, because it will play some role later in this section (see Proof of Lemma \ref{cusp}). We refer to \cite[page 8]{AignerZiegler2004} for a proof.

\begin{theorem}[Legendre]
Let $n\gs 0$ be a non-negative integer, let $p$ be a prime number, and let $\sigma_p (n)$ be the sum of the base $p$ digits of $n$. Then,
\[
v_p \left(n!\right)=\frac{n-\sigma_p (n)}{p-1},
\]
where, given any non-negative integer $m\gs 0$,
\[
v_p\left(m\right):=\max\{t\gs 0:\ p^t\mid m\}.
\]
\end{theorem}

\begin{lemma}\label{issue of divisibility}
Let $p\neq 2$ be a prime. Then,
\[
\lambda :=\frac{(p^2-1)!}{\left(\left(\frac{p^2-1}{2}\right)!\right)^2}\neq 0\pmod p.
\]
\end{lemma}

\begin{proof}
On one hand, $p^2-1=(p-1)(1+p)$ is the base $p$ expansion of $p^2-1$; on the other hand, since $p\neq 2$ it follows that
\[
\frac{p^2-1}{2}=\left(\frac{p-1}{2}\right)(1+p)
\]
is the base $p$ expansion of $(p^2-1)/2$. Keeping in mind these two facts it follows, using Legendre's Theorem, that
\begin{align*}
v_p \left(\lambda\right)& =v_p \left((p^2-1)!\right)-2v_p\left(\left(\frac{p^2-1}{2}\right)!\right)\\ & =\frac{p^2-1-2(p-1)}{p-1}-2\cdot\left(\frac{\frac{p^2-1}{2}-2\left(\frac{p-1}{2}\right)}{p-1}\right)=p-1-(p-1)=0,
\end{align*}
hence $p$ does not divide $\lambda$ and therefore we can ensure that $\lambda\neq 0\pmod p$.
\end{proof}

\begin{lemma}\label{cusp}
Let $p\neq 2,3$ be a prime, and let $f=y^2z-x^3\in\mathbb{F}_p [x,y,z]$. Then, $I_1 = I_2 = (x,y)$. In particular, $f$ has level two.
\end{lemma}
\begin{proof} Consider the expansion
\[
\ds f^{p^2-1} = \sum_{i=0}^{p^2-1} \frac{(p^2-1)!}{i!(p^2-1-i)!} y^{2i}z^ix^{3(p^2-1-i)}.
\]
For $i=(p^2-1)/2$ we obtain the monomial
\[
\ds \lambda y^{p^2-1}z^{(p^2-1)/2}x^{3(p^2-1)/2} = \lambda x^{p^2} \cdot \left(x^{(p^2-3)/2}y^{p^2-1}z^{(p^2-1)/2}\right),
\]
where $\lambda = \frac{(p^2-1)!}{\left(\left(\frac{p^2-1}{2}\right)!\right)^2} \ne 0$ (indeed, this follows by Lemma \ref{issue of divisibility}). Because of the term $z^i$ in the expansion above, the choice $i=(p^2-1)/2$ is clearly the only one that gives the monomial $x^{(p^2-3)/2}y^{p^2-1}z^{(p^2-1)/2}$ of the basis of $R$ as an $R^{p^2}$-module. Therefore $c\left(\frac{p^2-3}{2},p^2-1,\frac{p^2-1}{2}\right) = \lambda^{1/p^2}x = \lambda x$, and thus $x \in I_2$. In addition, by Lemma \ref{y} we always have $y \in I_2$. Therefore $(x,y) \subseteq I_2$. On the other hand, consider the expansion
\[
\ds f^{p-1} = \sum_{j=0}^{p-1} \frac{(p-1)!}{j!(p^2-1-j)!} y^{2j}z^jx^{3(p-1-j)}.
\] 
We claim that either $2j \gs p$ or $3(p-1-j) \gs p$. In fact, suppose $j < p/2$, or equivalently $j \ls (p-1)/2$, because $j$ is an integer. Then $3(p-1-j) \gs 3(p-1)/2 \gs p$ since $p \gs 5$ by assumption. This shows that all the coefficients $c(r,s,t)^p$ in the expansion of $f^{p-1} = \sum_{0 \ls r,s,t \ls p-1} c(r,s,t)^p x^ry^sz^t$ are contained in $(x,y)^{[p]}$, and thus $I_1 = (c(r,s,t) \ \mid \ 0 \ls r,s,t \ls p-1) \subseteq (x,y)$. Therefore the Lemma follows from the chain of inclusions $(x,y) \subseteq I_2 \subseteq I_1 \subseteq (x,y)$.
\end{proof}
\begin{lemma} \label{a_b_zero} Let $p \ne 2,3$ be a prime and let $\CC$ be a supersingular elliptic curve defined by $f(x,y,z) = y^2z-x^3+axz^2+bz^3$. If either $a=0$ or $b=0$, then $f$ has level two.
\end{lemma}
\begin{proof} Notice that, since $\CC$ is supersingular, we have $\FF_p[x,y,z] = R \supsetneq I_1 \supseteq I_2$, and we want to show that $I_1 = I_2$. Also, by Lemma \ref{y} we always have $y\in I_2$. If $a=b=0$, then Lemma \ref{cusp} ensures that $I_1=I_2 = (x,y)$. Now assume $a\ne 0$ and $b=0$. We claim that $c\left(\frac{p^2-1}{2},p^2-1,\frac{p^2-3}{2}\right) = \mu z$, for some $\mu \ne 0$. A monomial in the expansion of $f^{p^2-1}$ will have the form $(y^2z)^h(-x^3)^i(axz^2)^j$ for $h+i+j=p^2-1$. Looking at the terms involving $x^{\frac{p^2-1}{2}}y^{p^2-1}z^{\frac{p^2-3}{2}}$ the possibilities are
\[
\ds y^{2h}x^{3i+j} z^{2j+h} = \left\{
\begin{array}{ll}
x^{p^2} \cdot x^{\frac{p^2-1}{2}}y^{p^2-1}z^{\frac{p^2-3}{2}} \\
y^{p^2} \cdot x^{\frac{p^2-1}{2}}y^{p^2-1}z^{\frac{p^2-3}{2}} \\
z^{p^2} \cdot x^{\frac{p^2-1}{2}}y^{p^2-1}z^{\frac{p^2-3}{2}}
\end{array}
\right.
\]
The second case is not possible, since $2p^2-1$ is not even and hence here is no $h$ that gives it. This forces $h=(p^2-1)/2$, and hence the first case is also not possible, since for such an $h$ the exponent of $z$ must be at least $(p^2-1)/2$. Canceling $y^{p^2-1}z^{\frac{p^2-1}{2}}$ we are left with
\[
\ds x^{3i} z^{2j} = x^{\frac{p^2-1}{2}}z^{p^2-1}
\]
This implies that $j=\frac{p^2-1}{2}$ and $i=0$. Therefore the coefficient of $x^{\frac{p^2-1}{2}}y^{p^2-1}z^{\frac{p^2-3}{2}}$ in the expansion of $f^{p^2-1}$ is
\[
\ds c\left(\frac{p^2-1}{2},p^2-1,\frac{p^2-3}{2}\right)^{p^2} =a^{\frac{p^2-1}{2}} z^{p^2} = (\mu z)^{p^2},
\]
where $\mu = \left(a^{\frac{p^2-1}{2}}\right)^{1/p^2} \ne 0$. This shows the claim. With similar considerations, one can see that 
\[
\ds c\left(\frac{p^2-3}{2},p^2-1,\frac{p^2-1}{2}\right) = x.
\]
This gives $(x,y,z) \subseteq I_2 \subseteq I_1 \subsetneq R$, which forces $I_1 = I_2$ because $I_1$ is a proper homogeneous ideal of $R=\mathbb{F}_p [x,y,z]$, hence $I_1\subseteq\left(x,y,z\right)$. Finally, if $a=0$ and $b \ne 0$ the arguments are completely analogous to the case $a\ne 0$, $b=0$. Here we get
\[
\ds c\left(\frac{p^2-3}{2},p^2-1,\frac{p^2-1}{2}\right) = x \ \ \ \mbox{ and } \ \ \ c\left(0,p^2-1,p^2-2\right) = \lambda z
\]
for some $\lambda \ne 0$. Therefore $(x,y,z) \subseteq I_2 \subseteq I_1 \subsetneq R$, which once again forces $I_1 = I_2$.
\end{proof}

\subsection{Main result}
Next statement is the main result of this section.
\begin{theorem} \label{elliptic_curves} Let $p \in \ZZ$ be a prime number and let $\CC \subseteq \PP^2_{\FF_p}$ be an elliptic curve defined by a cubic $f(x,y,z) \in \FF_p[x,y,z]$. Then
\begin{enumerate}[(i)]
  \setlength{\itemsep}{2pt}
  \setlength{\parskip}{0pt}
  \setlength{\parsep}{0pt}
\item \label{1} $\CC$ is ordinary if and only if $f$ has level one.
\item \label{2} $\CC$ is supersingular if and only if $f$ has level two.
\end{enumerate}
\end{theorem}
\begin{proof} Part (\ref{1}) follows from Lemma \ref{sqfpowers}, which also shows that if $f$ has level at least two, then $\CC$ is supersingular. So it is left to show that if $\CC$ is supersingular, then $f$ has level precisely equal to two. By Proposition \ref{char_2_3} we only have to consider the cases where $p\ne 2,3$, thus we can assume that $f$ is of the form
\[
\ds f(x,y,z) = y^2z - x^3 + axz^2 + bz^3.
\]
for $a,b \in \FF_p$. Furthermore, by Lemma \ref{a_b_zero} we can assume that $ab \ne 0$. First we claim that $c\left(\frac{p^2-3}{2},p^2-1,\frac{p^2-1}{2}\right) = x+ \lambda z$, where
\[
\ds \lambda = \sum_{i=0}^{\frac{p^2-7}{6}} \frac{(p^2-1)!}{\left(\frac{p^2-1}{2}\right)!i!\left(\frac{p^2-3}{2}-3i\right)!(2i+1)!} (-1)^i a^{\frac{p^2-3}{2} - 3i} b^{2i+1} \in \FF_p.
\]
In fact a general monomial in the expansion of $f^{p^2-1}$ will have the form $(y^2z)^h(-x^3)^i(axz^2)^j(bz^3)^k$, and by looking at terms that involve $x^{\frac{p^2-3}{2}}y^{p^2-1}z^{\frac{p^2-1}{2}}$, by degree considerations we have three possibilities:
\[
\ds y^{2h}x^{3i+j} z^{h+2j+3k} = \left\{
\begin{array}{ll}
x^{p^2} \cdot x^{\frac{p^2-3}{2}}y^{p^2-1}z^{\frac{p^2-1}{2}} \\
y^{p^2} \cdot x^{\frac{p^2-3}{2}}y^{p^2-1}z^{\frac{p^2-1}{2}} \\
z^{p^2} \cdot x^{\frac{p^2-3}{2}}y^{p^2-1}z^{\frac{p^2-1}{2}}
\end{array}
\right.
\]
The second case cannot be realized since $2p^2-1$ is not even, hence we necessarily have $h=\frac{p^2-1}{2}$. This leaves two cases:
\[
\ds x^{3i+j} z^{2j+3k} = \left\{
\begin{array}{ll}
x^{p^2} \cdot x^{\frac{p^2-3}{2}} \\
z^{p^2} \cdot x^{\frac{p^2-3}{2}}
\end{array}
\right.
\]
The first one can only happen when $h=i=\frac{p^2-1}{2}$ and $j=k=0$, giving the monomial
\[
\ds (y^2z)^{\frac{p^2-1}{2}}(-x^3)^{\frac{p^2-1}{2}} = (-1)^{\frac{p^2-1}{2}}x^{p^2} \cdot x^{\frac{p^2-3}{2}}y^{p^2-1}z^{\frac{p^2-1}{2}} = x^{p^2} \cdot x^{\frac{p^2-3}{2}}y^{p^2-1}z^{\frac{p^2-1}{2}},
\]
because $p^2 \equiv 1$ (mod $4$) . For the second case $i,j$ and $k$ will have to satisfy
\begin{eqnarray*}
\left\{
\begin{array}{l}
j=  \frac{p^2-3}{2}-3i \\ \\
k=2i+1
\end{array}
\right.
\end{eqnarray*}
As both $j$ and $k$ must be non-negative, we have 
\[
\ds 0 \ls i \ls \left\lfloor \frac{p^2-3}{6} \right\rfloor = \frac{p^2-7}{6},
\]
because $p^2 \equiv 1$ (mod $6$). Furthermore, the coefficient of $(y^2z)^h(-x^3)^i(axz^2)^j(bz^3)^k$ under these conditions is precisely
\[
\ds \frac{(p^2-1)!}{\left(\frac{p^2-1}{2}\right)!i!\left(\frac{p^2-3}{2}-3i\right)!(2i+1)!} (-1)^i a^{\frac{p^2-3}{2} - 3i} b^{2i+1},
\]
proving that
\[
\ds c\left(\frac{p^2-3}{2},p^2-1,\frac{p^2-1}{2}\right)^{p^2} = x^{p^2} + \lambda z^{p^2} = (x+\lambda^{1/p^2}z)^{p^2},
\]
for $\lambda$ as above. Since we are working over $\FF_p$, we finally have $\lambda^{1/p^2} = \lambda$, showing the claim. With similar arguments, one can see that
\begin{eqnarray*}
\begin{array}{l}
c\left(\frac{p^2-1}{2},p^2-3,\frac{p^2+1}{2}\right) = ax+\mu z, \\ \\
c\left(\frac{p^2-7}{2},p^2-1,\frac{p^2+3}{2}\right) = -ax+\tau z, \\ \\
c\left(0,p^2-1,p^2-2\right) = \theta x+b^{\frac{p^2-1}{2}} z,
\end{array}
\end{eqnarray*}
where $\theta \in \FF_p$,
\[
\ds \mu = \sum_{i=0}^{\frac{p^2-1}{6}} \frac{(p^2-1)!}{\left(\frac{p^2-3}{2}\right)!i!\left(\frac{p^2-1}{2}-3i\right)!(2i+1)!} (-1)^i a^{\frac{p^2-1}{2} - 3i} b^{2i+1} \in \FF_p ,
\]
and
\[
\ds \tau = \sum_{i=0}^{\frac{p^2-7}{6}} \frac{(p^2-1)!}{\left(\frac{p^2-1}{2}\right)!i!\left(\frac{p^2-7}{2}-3i\right)!(2i+3)!} (-1)^i a^{\frac{p^2-7}{2} - 3i} b^{2i+3} \in \FF_p.
\]
Since $I_2 = (c(r,s,t) \mid 0 \ls r,s,t \ls p^2-1)$, in particular we have 
\[
(x+\lambda z,ax+\mu z, -ax + \tau z, \theta x+b^{\frac{p^2-1}{2}} z) \subseteq I_2.
\]
\begin{claim*} The following matrix has full rank:
\begin{eqnarray*} \left[\begin{matrix} a & \mu \\ 1 & \lambda \\ -a & \tau \\ \theta & b^{\frac{p^2-1}{2}} \end{matrix} \right].
\end{eqnarray*}
\end{claim*}
In fact if $\det \left[\begin{matrix} a & \mu \\ 1 & \lambda \end{matrix}\right] \ne 0$ then we are done, otherwise we have $\mu = a\lambda$. Note that
\begin{eqnarray*}
\begin{array}{c}
\ds \mu = \sum_{i=0}^{\frac{p^2-1}{6}} \frac{(p^2-1)!}{\left(\frac{p^2-3}{2}\right)!i!\left(\frac{p^2-1}{2}-3i\right)!(2i+1)!} (-1)^i a^{\frac{p^2-1}{2} - 3i} b^{2i+1} \\ \\
= \ds \left[\sum_{i=0}^{\frac{p^2-7}{6}} \frac{(p^2-1)!}{\left(\frac{p^2-3}{2}\right)!i!\left(\frac{p^2-1}{2}-3i\right)!(2i+1)!} (-1)^i a^{\frac{p^2-1}{2} - 3i} b^{2i+1}\right] + (-1)^{\frac{p^2-1}{6}}\frac{(p^2-1)!}{\left(\frac{p^2-3}{2}\right)!\left(\frac{p^2-1}{6}\right)!\left(\frac{p^2+2}{3}\right)!} b^{\frac{p^2+2}{3}} \\ \\
= \ds  \left[\sum_{i=0}^{\frac{p^2-7}{6}} \frac{(p^2-1)!}{\left(\frac{p^2-3}{2}\right)!i!\left(\frac{p^2-1}{2}-3i\right)!(2i+1)!} (-1)^i a^{\frac{p^2-1}{2} - 3i} b^{2i+1}\right] + \frac{3(p^2-1)!}{\left(\frac{p^2-1}{2}\right)!\left(\frac{p^2-7}{6}\right)!\left(\frac{p^2+2}{3}\right)!} b^{\frac{p^2+2}{3}},
\end{array}
\end{eqnarray*}
where the last equality comes from the fact that $p^2 \equiv 1$ (mod $12$) and by rearranging the binomial coefficients. Also,
\[
\ds a\lambda = \sum_{i=0}^{\frac{p^2-7}{6}} \frac{(p^2-1)!}{\left(\frac{p^2-1}{2}\right)!i!\left(\frac{p^2-3}{2}-3i\right)!(2i+1)!} (-1)^i a^{\frac{p^2-1}{2} - 3i} b^{2i+1}.
\] 
Using that $\mu = a \lambda$, we get
\begin{eqnarray*}
\begin{array}{c}
\ds \frac{3(p^2-1)!}{\left(\frac{p^2-1}{2}\right)!\left(\frac{p^2-7}{6}\right)!\left(\frac{p^2+2}{3}\right)!} b^{\frac{p^2+2}{3}}\\ \\
= \ds \sum_{i=1}^{\frac{p^2-7}{6}} \left[\frac{1}{\frac{p^2-1}{2}} - \frac{1}{\frac{p^2-1}{2}-3i}\right] \frac{(p^2-1)!}{\left(\frac{p^2-3}{2}\right)!i!\left(\frac{p^2-3}{2}-3i\right)!(2i+1)!} (-1)^i a^{\frac{p^2-1}{2} - 3i} b^{2i+1} \\ \\
\ds = \sum_{i=1}^{\frac{p^2-7}{6}}  \frac{-3(p^2-1)!}{\left(\frac{p^2-1}{2}\right)!(i-1)!\left(\frac{p^2-1}{2}-3i\right)!(2i+1)!} (-1)^i a^{\frac{p^2-1}{2} - 3i} b^{2i+1} \\ \\
= \ds 3\sum_{i=0}^{\frac{p^2-13}{6}}  \frac{(p^2-1)!}{\left(\frac{p^2-1}{2}\right)!i!\left(\frac{p^2-7}{2}-3i\right)!(2i+3)!} (-1)^i a^{\frac{p^2-7}{2} - 3i} b^{2i+3},
\end{array}
\end{eqnarray*}
where the last equality comes from reindexing the sum. Since $3$ is invertible in $\FF_p$ we get that
\begin{eqnarray*}
\begin{array}{c}
\ds 0 = \left[\sum_{i=0}^{\frac{p^2-13}{6}}  \frac{(p^2-1)!}{\left(\frac{p^2-1}{2}\right)!i!\left(\frac{p^2-7}{2}-3i\right)!(2i+3)!} (-1)^i a^{\frac{p^2-7}{2} - 3i} b^{2i+3}\right] - \left[\frac{(p^2-1)!}{\left(\frac{p^2-1}{2}\right)!\left(\frac{p^2-7}{6}\right)!\left(\frac{p^2+2}{3}\right)!} b^{\frac{p^2+2}{3}}\right] \\ \\
= \ds \sum_{i=0}^{\frac{p^2-7}{6}}  \frac{(p^2-1)!}{\left(\frac{p^2-1}{2}\right)!i!\left(\frac{p^2-7}{2}-3i\right)!(2i+3)!} (-1)^i a^{\frac{p^2-7}{2} - 3i} b^{2i+3}
\end{array}
\end{eqnarray*}
because $p^2 \equiv 1$ (mod $12$), hence $(-1)^{\frac{p^2-7}{6}} = -1$. But the latter is precisely $\tau$, and this argument shows that if $a \lambda = \mu$, then $\tau = 0$. Therefore either $a \lambda \ne \mu$ or
\[
\det \left[\begin{matrix} -a & \tau \\ \theta & b^{\frac{p^2-1}{2}} \end{matrix} \right] = \det \left[\begin{matrix} -a & 0 \\ \theta & b^{\frac{p^2-1}{2}} \end{matrix} \right] = -ab^{\frac{p^2-1}{2}} \ne 0.
\]
Hence the matrix has rank two, and the Claim follows. \\
\noindent But this shows that there are linear combinations of $x+\lambda z,ax+\mu z, -ax + \tau z$ and $\theta x+b^{\frac{p^2-1}{2}} z$ that produce $x$ and $z$, that is $(x,z) \subseteq I_2$. By Lemma \ref{y} we always have that $y \in I_2$. Therefore
\[
(x,y,z) \subseteq I_2 \subseteq I_1 \subsetneq R,
\]
implying that $I_1 = I_2$ and hence that $f$ has level two.
\end{proof}

\section{A Macaulay2 session}
The purpose of this section is to explain, through a Macaulay2 session, how the algorithm introduced in Section \ref{the algorithm described in detail} works in specific examples.

We begin clearing the previous input and loading our scripts.
\begin{verbatim}
clearAll;
load "differentialOperator.m2";
\end{verbatim}
We fix the polynomial ring that we will use throughout the following examples.
\begin{verbatim}
p=2;
F=ZZ/p;
R=F[x,y,z,w];
\end{verbatim}
The first example illustrates a particular case of Theorem \ref{monomial}.
\begin{verbatim}
i6 : f=x^3*y^5*z^7*w^4;

i7 : L=differentialOperatorLevel(f);

i8 : L

                2 4 6 3
o8 = (4, ideal(x y z w ))
\end{verbatim}
This means that, in this case, $f$ has level $4$ and that $I_4 (f^{2^4-1})=\left(x^2 y^4 z^6 w^3\right)$. Now, we produce a differential operator $\delta$ of level $4$ such that $\delta (1/f)=1/f^2$.
\begin{verbatim}
i7 : DifferentialOperator(f)

o7 = | x10y6z2w8 d_0^15d_1^15d_2^15d_3^15x2y4z6w3 |
\end{verbatim}
As the reader will note, the output is a row matrix; it means that $\delta$ turns out to be
\[
x^{10}y^6z^2w^8\cdot D_{x,2^4-1}D_{y,2^4-1}D_{z,2^4-1}D_{w,2^4-1}\cdot x^2y^4z^6w^3.
\]
Our next aim is to illustrate a particular case of Corollary \ref{linforms}.
\begin{verbatim}
ii8 : f=x^3*(x+y)^5*(x+y+z)^7*(x+y+z+w)^4;

ii9 : L=differentialOperatorLevel(f);

ii10 : first L

oo10 = 4

\end{verbatim}
Now, a particular case of Proposition \ref{sqfree}.
\begin{verbatim}
ii13 : f=x^2+y^2+z^3+x*y*z*w;

ii14 : L=differentialOperatorLevel(f);

ii15 : L

oo15 = (1, ideal 1)
\end{verbatim}
This means that $f$ has level one. Now, we produce the corresponding differential operator.
\begin{verbatim}
i7 : DifferentialOperator(f)

o7 = | 1 d_0d_1d_2d_3 |
\end{verbatim}
It means that the differential operator produced in this case is $D_{x,1}D_{y,1}D_{z,1}D_{w,1}$.

The following computation may be regarded as a particular case of Proposition \ref{allsqfree}.
\begin{verbatim}
i6 : f=x*w-y*z;

i7 : DifferentialOperator(f)

o7 = | 1 d_0d_1d_2d_3yz |
\end{verbatim}
It means that, in this case, the differential operator produced is $D_{x,1}D_{y,1}D_{z,1}D_{w,1}\cdot yz$.

Next, a homogeneous quadric (cf.\,Proposition \ref{quadratic forms}).
\begin{verbatim}
i6 : f=x^2+y^2+x*y+z^2+w^2;

i7 : DifferentialOperator(f)

o7 = | 1 d_0d_1d_2d_3zw |
\end{verbatim}
We finish with a homogeneous cubic.
\begin{verbatim}
ii12 : f=x^3+y^3+z^3+w^3;

ii13 : L=differentialOperatorLevel(f);

ii14 : L

oo14 = (2, ideal (w, z, y, x))

oo14 : Sequence

ii15 : DifferentialOperator(f)

oo15 = | w2 d_0^3d_1^3d_2^3d_3^3x3z3w |
     | z2 d_0^3d_1^3d_2^3d_3^3x3zw3 |
     | y2 d_0^3d_1^3d_2^3d_3^3yz3w3 |
     | x2 d_0^3d_1^3d_2^3d_3^3xy3z3 |

\end{verbatim}
This means that our differential operator in this case turns out to be
\begin{align*}
& w^2\cdot D_{x,3}D_{y,3}D_{z,3}D_{w,3}\cdot x^3z^3 w\quad +z^2\cdot D_{x,3}D_{y,3}D_{z,3}D_{w,3}\cdot x^3zw^3\\
& +y^2\cdot D_{x,3}D_{y,3}D_{z,3}D_{w,3}\cdot yz^3w^3\quad +x^2\cdot D_{x,3}D_{y,3}D_{z,3}D_{w,3}\cdot xy^3z^3.
\end{align*}

\section{The code of the algorithm}
The aim of this final section is to show our implementation in Macaulay2 of the algorithm described in Section \ref{the algorithm described in detail} of this manuscript; the whole code can be found in \cite{BoixDeStefaniVanzoM2}. Throughout this section, $R=\mathbb{F}_p [x_1,\ldots ,x_d]$ will be the polynomial ring with $d$ variables with coefficients in the field $\mathbb{F}_p$.

First of all, we write down the code of a procedure which, given an ideal $I$ of $R$, return as output $I^{[p^e]}$, i.e., the ideal generated by all the $p^e$th powers of elements in $I$. The method below is based on code written by M.\,Katzman and included, among other places, in \cite{FsplittingM2}.

\begin{verbatim}
frobeniusPower(Ideal,ZZ) := (I,e) ->(
R:=ring I;
p:=char R;
local u;
local answer;
G:=first entries gens I;
if (#G==0) then
{
     answer=ideal(0_R);
}
else
{
     N:=p^e;
     answer=ideal(apply(G, u->u^N));
};
answer
);
\end{verbatim}
Now, we exhibit the code of a function which, given ideals $A,B$ of $R$, produces as output the ideal $I_e (A)+B$. For our purposes in this manuscript, $B=(0)$ and $A$ is a principal ideal. Once again, this is based on code written by Katzman and included in \cite{FsplittingM2}.

\begin{verbatim}
ethRoot(Ideal,Ideal,ZZ):= (A,B,e) ->(
R:=ring(A);
pp:=char(R);
F:=coefficientRing(R);
n:=rank source vars(R);
vv:=first entries vars(R);
R1:=F[vv, Y_1..Y_n, MonomialOrder=>ProductOrder{n,n},MonomialSize=>16];
q:=pp^e;
J0:=apply(1..n, i->Y_i-substitute(vv#(i-1)^q,R1));
S:=toList apply(1..n, i->Y_i=>substitute(vv#(i-1),R1));
G:=first entries compress( (gens substitute(A,R1))%gens(ideal(J0)) );
L:=ideal 0_R1;
apply(G, t->
{
    L=L+ideal((coefficients(t,Variables=>vv))#1);
});
L1:=L+substitute(B,R1);
L2:=mingens L1;
L3:=first entries L2;
L4:=apply(L3, t->substitute(t,S));
use(R);
substitute(ideal L4,R)
);
\end{verbatim}
Next, we provide the code of our implementation of Algorithm \ref{calculo del nivel}. Namely, given $f\in R$, the procedure below gives as output the pair $\left(e, I_e\left(f^{p^e-1}\right)\right)$, where $e$ is the level of $f$, and $I_e\left(f^{p^e-1}\right)$ is the ideal where the chain \eqref{cadena incordio} stabilizes. As the reader can easily point out, this method is just turning Theorem \ref{first step of the algorithm} into an algorithm.

\begin{verbatim}
differentialOperatorLevel(RingElement):=(f) ->(
R:=ring(f);
p:=char(R);
local J;
local I;
local e;
e=0;
flag:=true;
local q;
local N;
while (flag) do
{
	e=e+1;
	q=p^e;
	N=q-1;
	I=ethRoot(ideal(f^N),ideal(0_R),e);
	J=frobeniusPower(I,e);
	N=q-p;
	if ((f^N)% J==0) then flag=false;
};
(e,I)
);
\end{verbatim}
Now, let $x=x_i$ (for some $1\ls i\ls d$), $n\gs 0$, and $f\in R$. The below procedure returns as output
\[
\frac{1}{n!}\frac{\partial^n f}{\partial x^n}.
\]
It is worth noting that, in some intermediate step of this method, we have to lift our data to characteristic zero in order to avoid problems with the calculation of $1/n!$.

\begin{verbatim}
DiffOperator(RingElement,ZZ,RingElement):=(el,numb,funct)->
(
     R:=ring(el);
     vv:=first entries vars(R);
     S:=QQ[vv];
     funct1:=substitute(funct,S);
     el=substitute(el,S);
     for i to numb-1 do
	   funct1=diff(el,funct1);
     funct1=1/(numb!)*funct1;
     funct=substitute(funct1,R);
     el=substitute(el,R);
     use R;
     return funct;
);
\end{verbatim}
Next, we provide a method which, given a monomial $\mathbf{x}^{\alpha}=x_1^{a_1}\cdots x_d^{a_d}$ with $0\ls a_i\ls p-1$ for all $i$, returns as output the differential operator $\delta_{\alpha}\in\mathcal{D}_R^{(e)}$ where, for any other monomial $\mathbf{x}^{\beta}=x_1^{b_1}\cdots x_d^{b_d}$ with $0\ls b_i\ls p-1$ for all $i$, $\delta_{\alpha}$ acts in the following way:
\[
\delta_{\alpha}\left(\mathbf{x}^{\beta}\right)=\begin{cases} 1,\text{ if }\alpha=\beta,\\ 0,\text{ otherwise.}\end{cases}
\]
As the reader can easily point out, the below method is just turning Claim \ref{Dirac delta claim} into an algorithm.
\begin{verbatim}
DeltaOperator(RingElement,ZZ):=(el,pe)->
(
     R:=ring(el);
     indet:=vars(R);
     for i to numColumns(indet)-1 do
     (
	 deg:=degree(indet_(0,i),el);	              			   		       
	 listdiff_i=pe-1-deg;
     );	 
     return listdiff;
);
\end{verbatim}
The next two methods are quite technical; however, both are necessary in order to avoid problems during the execution of our main procedure, which we are almost ready to describe.
\begin{verbatim}
checkcondition=method();

--- checkcondition finds a random monomial in startingpol//pol with all the

--- variables in degree <p^e

checkcondition(ZZ,RingElement,Ideal,ZZ):=(pe,startingpol,J,indexvar) ->
(    
     pol:=(first entries gens J)_indexvar;
     genJ=first entries gens J;
     R:=ring(pol);
     var:=vars R;
     nvars:=numColumns var;
     for i to numgens J-1 do
     	  sett_i=set first entries monomials(startingpol//genJ_i);
     supportpol= first entries monomials (startingpol//pol);
     contat=0;
     for i to #supportpol-1 do
     (
	  flag=true;
	  for k to nvars-1 do
	      if degree(var_(0,k),supportpol_i)>pe-1 then flag=false;
     	  if flag then 
              (
     		   rightsupport_contat=supportpol_i;
		   contat=contat+1;
	      );
     );
     correctmon=rightsupport_(random contat);     
     return correctmon;		   			   
);		         
-----------------------------------------

DifferentialAction=method();

--- DifferentialAction computes differential(element)

DifferentialAction(RingElement,RingElement):=(differential,element)->
(
     R:=ring(element);
     T:=ring(differential);
     var:=vars T;
     nvars:=numColumns var;
     total=0;
     diffmonomials=first entries monomials (differential);
     for i to #diffmonomials-1 do
     (
	  moltiplication=substitute(element,T);
	  for j to floor(nvars/2)-1 do
	     moltiplication=moltiplication*
	     var_(0,floor(nvars/2)+j)^(degree(var_(0,floor(nvars/2)+j),diffmonomials_i)); 
	  differentiation=coefficient(diffmonomials_i,differential)*moltiplication;
	  for j to floor(nvars/2)-1 do
	     differentiation=DiffOperator(var_(0,nvars-1-j),
		  degree(var_(0,floor(nvars/2)-1-j),diffmonomials_i),differentiation);
      total=total+differentiation;
      );
return total;
);
\end{verbatim}
We conclude showing our implementation of the algorithm described in Section \ref{the algorithm described in detail}, which is the main result of this paper.
\begin{verbatim}
DifferentialOperator(RingElement):=(f)->
(
     R:=ring(f);
     p:=char(R);
     F:=coefficientRing(R);
     local e;
     local J;
     (e,J)=differentialOperatorLevel(f);
     J=frobeniusPower(J,e);
     variable:=vars(R);
     variable1:=first entries variable;
     genJ:=first entries gens J;
     nvars:=numColumns vars R;
     powerf:=f^(p^e-1);
     T:=F[d_0..d_(nvars-1),variable1]; --creating ring of differentials
     varDelta:=first entries vars T;
     powerfT=substitute(powerf,T);
     use R;
     condition:=true;
     while condition do
     (
     for i to numgens J-1 do
     (
 	   listsupport=checkcondition(p^e,powerf,J,i);
	   listdiff=DeltaOperator(listsupport,p^e);
 	   delta1=1;
 	   delta2=1;
 	  for j to nvars-1 do
 	  ( 
      	       use T;
      	       delta1=delta1*varDelta_j^(p^e-1);
      	       expon2=listdiff_j;
      	       delta2=delta2*varDelta_(nvars+j)^expon2;
 	  );
          delta_i=delta1*delta2;
	  newgen_i=DifferentialAction(delta_i,powerfT);     
      );
      newgenmatr=matrix(newgen_0);
      for i from 1 to numgens J-1 do
          newgenmatr=newgenmatr|newgen_i;
      newgenmatrR=substitute(newgenmatr,R);
      if J==ideal(newgenmatrR) then
      (
      	   matalpha=f^(p^e-p)//newgenmatrR;
      	   matalpha=substitute(matalpha,T);
      	   for i to numgens J-1 do
               if i==0 then matrixdiff=matrix{{matalpha_(i,0),delta_i}}
               else matrixdiff=matrixdiff||matrix{{matalpha_(i,0),delta_i}};
      	  totale=0;
      	  for i to numgens J-1 do
               totale=totale+matalpha_(i,0)*DifferentialAction(delta_i,powerf);
      	  powerfp=substitute(f^(p^e-p),T);                          
	  return matrixdiff;
	  condition=false;
       );		   
      );
      use R;			      
);
\end{verbatim}

\section*{Acknowledgements}
This project started during the summer school PRAGMATIC 2014. The authors would like to thank Aldo Conca, Srikanth Iyengar, and Anurag Singh, for giving very interesting lectures and for sharing open problems. They also wish to thank Alfio Ragusa and all the organizers of PRAGMATIC 2014 for giving them the opportunity to attend the school. The authors also thank Josep \`Alvarez Montaner and, once again, Anurag Singh and Srikanth Iyengar, for many helpful suggestions concerning the material of this manuscript.

\bibliographystyle{alpha}
\bibliography{References}

\newcommand{\etalchar}[1]{$^{#1}$}
\def\cprime{$'$}
\begin{thebibliography}{{\`A}MBL05}

\bibitem[{\`A}MBL05]{AMBL}
J.~{\`A}lvarez~Montaner, M.~Blickle, and G.~Lyubeznik.
\newblock Generators of {$D$}-modules in positive characteristic.
\newblock {\em Math. Res. Lett.}, 12(4):459--473, 2005.

\bibitem[AZ04]{AignerZiegler2004}
M.~Aigner and G.~M. Ziegler.
\newblock {\em Proofs from {T}he {B}ook}.
\newblock Springer-Verlag, Berlin, third edition, 2004.
\newblock Including illustrations by Karl H. Hofmann.

\bibitem[BDSV14]{BoixDeStefaniVanzoM2}
A.F. Boix, A.~De~Stefani, and D.~Vanzo.
\newblock differentialoperator.m2: a {M}acaulay2 package for computing a
  differential operator which acts as the {F}robenius homomorphism.
\newblock Available at
  \href{http://atlas.mat.ub.edu/personals/aboix/research.html}{http://atlas.mat.ub.edu/personals/aboix/research.html},
  2014.

\bibitem[Ber72]{Bernstein1972}
I.~N. Bern{\v{s}}te{\u\i}n.
\newblock Analytic continuation of generalized functions with respect to a
  parameter.
\newblock {\em Funkcional. Anal. i Prilo\v zen.}, 6(4):26--40, 1972.

\bibitem[BMS08]{BMS}
M.~Blickle, M.~Musta{\c{t}}{\v{a}}, and K.~E. Smith.
\newblock Discreteness and rationality of {$F$}-thresholds.
\newblock {\em Michigan Math. J.}, 57:43--61, 2008.
\newblock Special volume in honor of Melvin Hochster.

\bibitem[B{\o}g95]{Bog}
R.~B{\o}gvad.
\newblock Some results on {D}-modules on {B}orel varieties in characteristic
  {$p>0$}.
\newblock {\em J. Algebra}, 173(3):638--667, 1995.

\bibitem[Can15]{Canton15}
E.~Canton.
\newblock A note on injectivity of {F}robenius on local cohomology of
  hypersurfaces.
\newblock Available at
  \href{http://arxiv.org/pdf/1502.03184v1}{http://arxiv.org/pdf/1502.03184v1},
  2015.

\bibitem[FST11]{FujinoSchwedeTakagi11}
O.~Fujino, K.~Schwede, and S.~Takagi.
\newblock Supplements to non-lc ideal sheaves.
\newblock In {\em Higher dimensional algebraic geometry}, RIMS K\^oky\^uroku
  Bessatsu, B24, pages 1--46. Res. Inst. Math. Sci. (RIMS), Kyoto, 2011.

\bibitem[Gro67]{EGA_4_4}
A.~Grothendieck.
\newblock \'{E}l\'ements de g\'eom\'etrie alg\'ebrique. {IV}. \'{E}tude locale
  des sch\'emas et des morphismes de sch\'emas {IV}.
\newblock {\em Inst. Hautes \'Etudes Sci. Publ. Math.}, (32), 1967.

\bibitem[GS]{M2}
D.~R. Grayson and M.~E. Stillman.
\newblock Macaulay2, a software system for research in algebraic geometry.
\newblock Available at \href{http://www.math.uiuc.edu/Macaulay2/}%
  {http://www.math.uiuc.edu/Macaulay2/}.

\bibitem[Hus87]{Husemoller}
D.~Husemoller.
\newblock {\em Elliptic curves}, volume 111 of {\em Grad. Texts in Math.}
\newblock Springer-Verlag, New York, 1987.
\newblock With an appendix by Ruth Lawrence.

\bibitem[ILL{\etalchar{+}}07]{TwentyFourHours}
S.~B. Iyengar, G.~J. Leuschke, A.~Leykin, C.~Miller, E.~Miller, A.~K. Singh,
  and U.~Walther.
\newblock {\em Twenty-four hours of local cohomology}, volume~87 of {\em Grad.
  Stud. Math.}
\newblock American Mathematical Society, Providence, RI, 2007.

\bibitem[KS12a]{KatzmanSchwede2012}
M.~Katzman and K.~Schwede.
\newblock An algorithm for computing compatibly {F}robenius split subvarieties.
\newblock {\em J. Symbolic Comput.}, 47(8):996--1008, 2012.

\bibitem[KS12b]{FsplittingM2}
M.~Katzman and K.~Schwede.
\newblock F{S}plitting, a {M}acaulay2 package implementing an algorithm for
  computing compatibly {F}robenius split subvarieties.
\newblock Available at
  \href{http://katzman.staff.shef.ac.uk/FSplitting/}{http://katzman.staff.shef.ac.uk/FSplitting/},
  2012.

\bibitem[Luc78]{Lucas}
E.~Lucas.
\newblock Sur les congruences des nombres eul\'eriens et les coefficients
  diff\'erentiels des functions trigonom\'etriques suivant un module premier.
\newblock {\em Bull. Soc. Math. France}, 6:49--54, 1878.

\bibitem[Ser73]{arithmeticcourse}
J.-P. Serre.
\newblock {\em A course in arithmetic}, volume~7 of {\em Grad. Texts in Math.}
\newblock Springer-Verlag, New York-Heidelberg, 1973.

\bibitem[Yek92]{Yek}
A.~Yekutieli.
\newblock An explicit construction of the {G}rothendieck residue complex.
\newblock {\em Ast\'erisque}, (208):127, 1992.
\newblock With an appendix by Pramathanath Sastry.

\end{thebibliography}
\end{document}